\documentclass[reqno,11pt]{amsart}

\usepackage{amsfonts}
\usepackage{amssymb, latexsym, amsmath, pb-diagram}
\usepackage{txfonts}
\usepackage{dsfont}
\usepackage{pstricks-add}
\usepackage{lamsarrow, pb-lams}
\usepackage{multicol}
\usepackage{bm}
\usepackage{longtable}
\usepackage[mathscr]{eucal}
\usepackage{epic,eepic}
\usepackage{xy}
\usepackage{graphicx}
\usepackage{epsfig}
\usepackage{epstopdf}
\xyoption{all}

\newtheorem{thm}{Theorem}[section]
\newtheorem{prop}{Proposition}[section]
\newtheorem{lem}{Lemma}[section]
\newtheorem{cor}{Corollary}[section]
\newtheorem{defn}{Definition}[section]

\newtheorem{rem}[thm]{Remark}

\begin{document}
\title[Topology of moment-angle manifolds]{The topology of the moment-angle manifolds\\----- on a conjecture of S.Gitler and S.L\'{o}pez}
\author[L.~Chen, F.~Fan \& X.~Wang]{Liman Chen, Feifei Fan and Xiangjun Wang}
\thanks{The authors are supported by NSFC grant No. 11261062 and SRFDP No.20120031110025}
\address{Liman Chen, School of Mathematical Sciences and LPMC, Nankai University, Tianjin 300071, P.~R.~China}
\email{chenlimanstar1@163.com}
\address{Feifei Fan, School of Mathematical Science and LPMC, Nankai University, Tianjin 300071, P.~R.~China}
\email{fanfeifei@mail.nankai.edu.cn}
\address{Xiangjun Wang, School of Mathematical Science and LPMC, Nankai University, Tianjin 300071, P.~R.~China}
\email{xjwang@nankai.edu.cn}
\date{}
\subjclass[2000]{Primary 22E46, 53C30}
\keywords{moment-angle manifolds, simple polytope, isotopy, regular embedding}
\maketitle

\begin{abstract}
In this paper, we study the topology of the moment-angle manifolds and prove a conjecture of S. Gitler and S. L\'{o}pez concerned with the behavior of the moment-angle manifold under the surgery `cutting off vertices' on a simple polytope.
Let $P$ be a simple polytope of dimension $n$ with $m$ facets and $P_{v}$ be a polytope obtained from
$P$ by cutting off one vertex $v$. Let $\mathcal {Z}=\mathcal {Z}(P)$ and $\mathcal {Z}_{v}=\mathcal {Z}(P_{v})$ be the corresponding moment-angle manifolds.
In \cite{[GL]} S. Gitler and S. L\'{o}pez
conjectured that: $\mathcal {Z}_{v}$ is diffeomorphic to
$\partial[(\mathcal {Z}-D^{n+m})\times D^{2}]\# \mathop{\#} \limits_{j=1}^{m-n} \binom{m-n}{j} (S^{j+2}\times S^{m+n-j-1})$,
and they have proved the conjecture in the case $m<3n$. In this paper we prove the conjecture in general case.
\end{abstract}

\section{Introduction}
\subsection{Background}
 The moment-angle manifold $\mathcal {Z}$ comes from two different ways:
\begin{enumerate}
\item The transverse intersections in $\mathds{C}^{n}$ of real quadrics of the form
$\sum \limits _{i=1}^{n}a_{i}|z_{i}|^{2}=0$ with the unit euclidean sphere of $\mathds{C}^{n}$.
\item An abstract construction from a simple polytope $P^{n}$ with $m$-facets (or a complex $K$).
\end{enumerate}

    The study of the first one led to the discovery of a new special class of compact
non-k\"{a}hler complex manifolds in the work of Lopez, Verjovsky and Meersseman (\cite{[LV]},\cite{[Me]},\cite{[MV]}), now known as the LV-M manifolds,
which helps us understand the topology of non-k\"{a}hler complex manifolds.

   The study of the second one is related to the quasitoric manifolds in the following way:
for every quasitoric manifold $\pi: M^{2n}\rightarrow P^{n}$, there is a principal
$T^{m-n}$-bundle $\mathcal {Z}\rightarrow M^{2n}$ whose composite map with $\pi$ makes $\mathcal {Z}$ a $T^{m}$-manifold with orbit space $P^{n}$.
The topology of the manifolds $\mathcal {Z}$ provids an effective tool for understanding inter-relations between algebraic
and combinatorial aspects such as the Stanley-Reisner rings, the subspace arrangements and the cubical complexs etc.(see \cite{[BP]}).

Following \cite{[BP]}, let $P^{n}$ be an n-dimensional simple polytope with m facets, and let $\mathcal{F}=\{F_{1},\ldots,F_{m}\}$ be the set of facets of $P^{n}$. For each facet $F_{i}\in\mathcal{F}$ denote by $T_{F_{i}}$ the one-dimensional coordinate subgroup of $T^{\mathcal{F}}\cong T^{m}$ corresponding to $F_{i}$. Then assign to every face $G$ the coordinate subtorus
$$T_{G}=\mathop{\Pi} \limits_{F_{i}\supset G}T_{F_{i}}\subset T^{\mathcal{F}}.$$
Note that dim$T_{G}$=codim $G$. Recall that for every point $q\in P^{n}$ we denote by $G(q)$ the unique face containing $q$ in the relative interior.

\begin{defn}
Define the moment-angle complex corresponding to $P^{n}$ as:
$$\mathcal {Z}(P)=(T^{\mathcal{F}}\times P^{n})/\sim,$$
where $(t_{1},p)\sim(t_{2},q)$ if and only if $p=q$ and $t_{1}t_{2}^{-1}\in T_{G(q)}$.
\end{defn}

From \cite{[BP]}, $\mathcal {Z}(P)$ is a smooth manifold.

However, according to Buchstaber-Panov\cite{[BP]}, there is another way to define the moment-angle complex corresponding to a simplicial complex and we will use this definition in this paper:
\begin{defn}
Let $K$ be a simplicial complex. we use $[m]=\{0,1,\ldots,m-1\}$ to represent the $m$ vertices
of the simplicial complex. Let $\sigma$ be a simplex in the complex $K$ and we use $|\sigma|$ to denote the number of the vertices of $\sigma$. Define
\[
D^{2|\sigma|}_{\sigma}\times T^{m-|\sigma|}_{\hat{\sigma}}=\{(z_{1}, z_2, \cdots, z_{m})\in (D^{2})^{m}: |z_{j}|=1 \; for \; j\notin \sigma\}.
\]
and define $\mathcal {Z}_{K}$ corresponding to $K$ as
\[
\mathcal {Z}_{K}=\bigcup \limits _{\sigma \in K}D^{2|\sigma|}_{\sigma}\times T^{m-|\sigma|}_{\hat{\sigma}}\subset (D^{2})^{m}.
\]
\end{defn}
Given a simple polytope $P$, then the dual of the boundary of $P$ is a simplicial complex $K_{P}$, which is defined to be a polytopal sphere. So we have two complexes: $\mathcal {Z}(P)$ and $\mathcal {Z}_{K_{P}}$. However, from \cite{[BP]}, $\mathcal {Z}_{K_{P}}$ is homeomorphic to $\mathcal {Z}(P)$, $\mathcal {Z}_{K_{P}}$ can be endowed with a smooth structure and $D^{2|\sigma|}_{\sigma}\times T^{m-|\sigma|}_{\hat{\sigma}}$ is a smooth tubular neighborhood of $\{0\}\times T_{\hat{\sigma}}$ in the smooth manifold $\mathcal {Z}_{K_{P}}$.

A fundamental problem in toric topology is to study the topology of the moment-angle manifolds $\mathcal {Z}_{K}$ corresponding to simplicial spheres $K$.
One way to do this is to consider the change of the moment-angle manifold $\mathcal {Z}$ after taking
some surgeries on the complex $K$. These surgeries include bistellar moves, cutting off faces, connected sum with other simplicial complexes , etc.\cite{[BP]}.
Obviously if we can precisely describe the behavior of the moment-angle manifold $\mathcal {Z}$ under some surgery on $K$, we will
get a better understanding of of the moment-angle manifolds. However, when studying the topology of the moment-angle manifolds corresponding to $PL$ spheres, we can start from one $PL$ sphere $K_{1}$ and make a sequence of surgeries on it to get another $PL$ sphere $K_{2}$, if we can clearly know the change of the topology of the corresponding moment-angle manifolds during these surgeries, we may construct the moment-angle manifold $\mathcal {Z}_{K_{2}}$ by making some simple `surgeries' on $\mathcal {Z}_{K_{1}}$.
Now we give the definition of bistellar $k-$move and consider the change of the topology of the moment-angle manifold after taking a bistellar $k-$move on $K$:
\begin{defn}
Let $K$ be an $(n-1)$ dimensional pure simplicial complex on the vertex set $[m]$, and let $\sigma \in K$ be an $(n-1-k)-$simplex $(1\leq k\leq n-2)$ such that $link_{K} \sigma$ is the boundary $\partial\tau$ of a $k-$simplex $\tau$ that is not a simplex of $K$. Then the operation $\chi_{\sigma}$ on $K$ defined by:
$$ \chi_{\sigma}(K):= (K-\sigma\ast\partial\tau)\cup (\partial\sigma\ast\tau)$$
is called a bistellar $k-$move. Obviously, a bistellar $0-$move is just the connected sum with the boundary of a $n$-simplex.
\end{defn}

However, there is a surgery called `cutting off vertex' on the simple polytope dual to a bistellar $0-$move on the polytopal sphere:
\begin{defn}
Let $P$ be a simple polytope of dimension $n$ with $m$ facets, which is the convex hull of finitely many vertices in $\mathds{R}^{n}$. For the vertex $v$, we can find a hyperplane $H(x) = \mathop{\sum} \limits_{i=1}^{n}a_{i}x_{i} = b$ satisfying that $H(v)>b$ \;for $v$ and $H(w)<b$ \;for every vertex $w\neq v$. The set $P\cap\{x|\;H(x)\leq b\}$ is a new simple polytope $P_{v}$, which is called to be obtained from $P$ by
cutting off vertex $v$.
\end{defn}
Let $K_{P}$ and $K_{P_{v}}$ be the duals of the boundary of $P$ and $P_{v}$, $\sigma$ be the maximal simplex in $K_{P}$ dual to the vertex $v$ of the simple polytope $P$. Then we have $K_{P_{v}}=K_{P}\#_{\sigma}\partial\triangle^{n}$ ($\triangle^{n}$ is the standard n-dimensional simplex, the choice of a maximal simplex in $\partial\triangle^{n}$ is irrelevant). From this, the surgery `cutting off vertex' on the simple polytope corresponds to the bistellar $0-$move on the dual polytopal sphere.

Two pure simplicial complexes are bistellarly equivalent if one is taken to another by a finite sequence of bistellar moves. It is easy to see that two bistellarly equivalent $PL$ manifolds are $PL$ homeomorphic. The following remarkable result shows that the converse is also true.

\begin{thm}
Two $PL$ manifolds are bistellarly equivalent if and only if they are $PL$ homeomorphic.
\end{thm}

We can also define a bistellar $k-$move on simple polytope by taking a bistellar $k-$move on the dual simplicial sphere. Then we have a similar theorem of Ewald:
\begin{thm}
Let $P$ be a simple polytope of dimension $n\geq3$. Then there is a sequence of simple polytopes $P_{1},\ldots,P_{m}$ such that $P_{1}=\triangle^{n}$ and $P_{m}=P$ and for $i=1,\ldots,m-1$, $K_{P_{i+1}}$ is obtained from $K_{P_{i}}$ by a bistellar $k-$move with $0\leq k\leq n-2$.
\end{thm}

So from these theorems, to study the topology of the moment-angle manifolds, we can consider the change of the topology of the moment-angle manifolds after taking some bistellar moves on the simplicial spheres or simple polytopes. while taking a bistellar $k-$move on simplicial sphere $K$, the moment-angle complexes corresponding to $\sigma\ast\partial\tau$ and $\partial\sigma\ast\tau$ are $D^{2(n-k)}_{\sigma}\times S^{2k+1}_{\tau}$ and $S^{2(n-k)-1}_{\sigma}\times D^{2(k+1)}_{\tau}$ respectively. By definition of moment-angle complex, we obtain
$$\mathcal {Z}_{\chi_{\sigma}(K)} = (\mathcal {Z}_{K}- T^{m-n-1}\times D_{\sigma}^{2(n-k)}\times S^{2k+1}_{\tau})\cup (T^{m-n-1}\times S^{2(n-k)-1}_{\sigma}\times D^{2(k+1)}_{\tau}),$$
where $T^{m-n-1}\times S^{2(n-k)-1}_{\sigma}\times D^{2(k+1)}_{\tau}$ is attached along boundary $T^{m-n-1}\times S_{\sigma}^{2(n-k)-1}\times S^{2k+1}_{\tau}$.
From this expression, we can't get too much information about the topology of $\mathcal {Z}_{\chi_{\sigma}(K)}$, we can't even calculate the cohomology of $\mathcal {Z}_{\chi_{\sigma}(K)}$ in general case. The topology of moment-angle manifolds is so complicated that we can't know clearly the change of the moment-angle manifolds during any bistellar move. But based on the following reasons, maybe we can precisely describe the behavior of the moment-angle manifold after taking a bistellar $0$- move on a simplicial sphere $K$ $($or cutting off vertex $v$ on a simple polytope $P$$)$:
\begin{enumerate}
\item  From  \cite{[BM]}, we know that moment-angle manifold $\mathcal {Z}(P_{v})$ only depends on the topology of the moment-angle manifold $\mathcal {Z}$, independent of the vertex $v$.
\item  we have the following theorem:
\begin{thm}[\cite{[Mc]}]
Let $P$ be a simple polytope obtained from the $k$-simplex by cutting off vertices for $l$ times.
Then the corresponding moment-angle manifold $\mathcal {Z}(P)$ is diffeomorphic to a connected sum of sphere products
\[\mathcal {Z}(P)\cong \overset{l}{\underset{j=1}\#}j\binom{l+1}{j+1}S^{j+2}\times S^{2k+l-j-1}.\]
\end{thm}
\item  S. Gitler and S.~L\'{o}pez \cite{[GL]} prove that when in the case of $m<3n$, $\mathcal {Z}(P_{v})$ is  diffeomorphic to $\partial\left[\left(\mathcal {Z}-D^{n+m}\right)\times D^{2}\right]\# \mathop{\#} \limits_{j=1}^{m-n} \binom{m-n}{j} \left(S^{j+2}\times S^{m+n-j-1}\right)$.
\item  we can calculate the cohomology of $\mathcal {Z}_{K\#\partial\triangle}$ using the method of Taylor resolution:
\begin{thm}[\cite{[FW]}]
The cohomology of the moment-angle manifold $\mathcal {Z}_{K\#\partial\triangle}$ is isomorphic to the cohomology of the manifold $\partial\left[\left(\mathcal {Z}-D^{n+m}\right)\times D^{2}\right]\# \mathop{\#} \limits_{j=1}^{m-n} \binom{m-n}{j} \left(S^{j+2}\times S^{m+n-j-1}\right)$.
\end{thm}
\end{enumerate}

Based on these reasons, S. Gitler and S.~L\'{o}pez conjectured that $\mathcal {Z}(P_{v})$ is diffeomorphic to
\[
\partial\left[\left(\mathcal {Z}-D^{n+m}\right)\times D^{2}\right]\# \mathop{\#} \limits_{j=1}^{m-n} \binom{m-n}{j} \left(S^{j+2}\times S^{m+n-j-1}\right)
\]

In this paper, we will prove the conjecture by constructing an isotopy of a submanifold in the moment-angle manifold.

\subsection{Main Idea of S.~Gitler and S.~L\'{o}pez}

In the case $m<3n$ (\cite{[GL]}), S.~Gitler and S.~L\'{o}pez firstly proved that $T^{m-n}_{\hat{\sigma}}\times \{0\}$ can be contracted to a point in $\mathcal {Z}$. Since $m<3n$, it is isotopic to a $m-n$-torus inside an
open disk in $\mathcal {Z}$. therefore,
\[
\mathcal {Z}-T_{\widehat{\sigma}}^{m-n}\times D_{\sigma}^{2n}\cong \mathcal {Z}-T^{m-n}\times D^{2n}\cong (\mathcal {Z}-D^{n+m})\cup (D^{n+m}-T^{m-n}\times D^{2n})
\]
and
\[\mathcal {Z}(P_{v})\simeq\partial[(\mathcal {Z}-D^{n+m})\times D^{2}]\# \partial[(S^{m+n}-T^{m-n}\times D^{2n})\times D^{2}].\]
Then they considered the manifold $\partial[(S^{m+n}-T^{m-n}\times D^{2n})\times D^{2}]$, where the $m-n$-torus embeds in $S^{m+n}$ by $S^{1}\times\cdots \times S^{1}\subseteq D^{2}\times\cdots \times D^{2}= D^{2(m-n)}\times \{0\}\subseteq D^{m+n}\subseteq S^{m+n}$.
They constructed spheres represent the Alexander dual homology of $T^{m-n}$ in
$S^{m+n}$. According to a corollary of $h$-cobordism theorem, they proved that
\[\partial[(S^{m+n}-T^{m-n}\times D^{2n})\times D^{2}]\simeq\mathop{\#} \limits_{j=1}^{m-n} \binom{m-n}{j} (S^{j+2}\times S^{m+n-j-1}).\]
So the conjecture is true in this case.

In this paper we will prove the conjecture in general case and now we state it as a theorem:

\begin{thm}
let $P$ be a simple polytope of dimension $n$ with $m$ facets and $P_{v}$ be a polytope obtained from
$P$ by cutting off one vertex $v$. Let $\mathcal {Z}=\mathcal {Z}(P)$ and $\mathcal {Z}_{v}=\mathcal {Z}(P_{v})$ be the corresponding moment-angle manifolds,
then $\mathcal {Z}_{v}$ is diffeomorphic to
\[
\partial[(\mathcal {Z}-D^{n+m})\times D^{2}]\# \mathop{\#} \limits_{j=1}^{m-n} \binom{m-n}{j} (S^{j+2}\times S^{m+n-j-1}).
\]
\end{thm}

\subsection{Main Idea of the Proof}
In section 2, we construct an isotopy of $T_{\hat{\sigma}}^{m-n}$ in $\mathcal {Z}$ to move it to the regular embedding
$T^{m-n}\subseteq D^{m-n+1}\subseteq D^{m+n}\subseteq \mathcal {Z}$, thus we prove the following:

\begin{prop}
$\mathcal {Z}_{v}$ is diffeomorphic to
\[
\partial[(\mathcal {Z}-D^{n+m})\times D^{2}]\mathop{\#}\partial[(S^{m+n}-T^{m-n}\times D^{2n})\times D^{2}],
\]
where $T^{m-n}\times D^{2n}$ is the regular embedding in $S^{m+n}$.
\end{prop}

 Then in section 3, we prove the following by induction:
\begin{prop}
$\partial[(S^{m+n}-T^{m-n}\times D^{2n})\times D^{2}]$ is diffeomorphic to
\[
 \mathop{\#} \limits_{j=1}^{m-n} \binom{m-n}{j} (S^{j+2}\times S^{m+n-j-1}),
\]
where $T^{m-n}\times D^{2n}$ is the regular embedding in $S^{m+n}$.
\end{prop}

Combining these two propositions, Theorem 1.5 is proved.

 The manifold $\partial[(\mathcal {Z}-D^{m+n})\times D^{2}]$ is diffeomorphic to $(\mathcal {Z}\times S^1-D^{m+n}\times S^1)\cup S^{m+n-1}\times D^{2}$,
which can be obtained by taking a $(m+n,1)$-type surgery on the manifold $\mathcal {Z}\times S^1$ (see \cite{[M]}).
\begin{cor}
Let $[\mathcal {Z}]$ and $[S^1]$ be the fundamental classes of $\mathcal {Z}$ and $S^1$ respectively. Then the cohomology of
$\partial[(\mathcal {Z}-D^{m+n})\times D^{2}]$ is isomorphic to
\[
H^*(\partial[(\mathcal {Z}-D^{m+n})\times D^{2}])\cong H^*(\mathcal {Z})\otimes H^*(S^1)/\{1\otimes [S^1], [\mathcal {Z}]\otimes 1\}
\]
as a ring.
\end{cor}

\section{Construct the Isotopy of $T^{m-n}_{\hat{\sigma}}\times {0}$ in $\mathcal {Z}$}
\subsection{Some Notations}
After cutting off a vertex $v$ on the simple polytope $P$, we obtain a new simple polytope $P_{v}$. Let $K_{P}$ and
$K_{P_{v}}$ be the duals of the boundary of $P$ and
$P_{v}$, $\sigma$ be the maximal simplex in $K_{P}$ dual to the vertex $v$ of the simple polytope $P$.
Then we have $K_{P_{v}}=K_{P}\#_{\sigma}\partial\triangle^{n}=(K_{P}-\sigma)\cup (\partial\triangle^{n}-\triangle^{n-1})$ ($\triangle^{n}$ is the standard n-dimensional simplex, $\triangle^{n-1}$ is a maximal simplex of $\partial\triangle^{n}$,
 the choice of a maximal simplex in $\partial\triangle^{n}$ is irrelevant).
By the definition, the moment-angle complex corresponding to $P$ (or $K_{P}$) is:
\[
\mathcal {Z}=\bigcup \limits _{\sigma \in K_{P}}D^{2|\sigma|}_{\sigma}\times T_{\widehat{\sigma}}^{m-|\sigma|}\subset (D^{2})^{m}.
\]
Denote $\hat{K_{1}}:=K_{P}-\{\sigma\}$ and $\hat{K_{2}}:=\partial\triangle^{n}-\sigma$. Then the moment-angle complex corresponding to the complex $\hat{K_{1}}$ and $\hat{K_{2}}$ are:
$$\mathcal {Z}_{\hat{K_{1}}}= \mathcal {Z}_{K_{P}}-T^{m-n}_{[m]-\sigma}\times D^{2n}_{\sigma}, \qquad \mathcal {Z}_{\hat{K_{2}}}= \mathcal {Z}_{\partial\triangle^{n}}-S^{1}\times D^{2n}_{\sigma}=S^{2n+1}-S^{1}\times D^{2n}_{\sigma}= D^{2}\times S^{2n-1}_{\sigma}$$
Then we can express the moment-angle complex corresponding to $P_{v}$ (or $K_{P_{v}}$) as follows (see 6.4 in \cite{[BP]}):
\begin{align*}
\mathcal {Z}_{v}&=(\mathcal {Z}\times S^{1}-T^{m-n}_{\widehat{\sigma}}\times D^{2n}_{\sigma}\times S^{1})
\cup_{T^{m-n}_{\widehat{\sigma}}\times S^{2n-1}_{\sigma}
\times S^{1}}T^{m-n}_{\widehat{\sigma}}\times S^{2n-1}_{\sigma}\times D^{2}\\
&\simeq \partial[(\mathcal {Z}-T^{m-n}_{\widehat{\sigma}}\times D_{\sigma}^{2n})\times D^{2}]\\
\end{align*}

We use $ S^{1}_{i}$ and $D^{2}_{i}$ to represent the coordinates corresponding to the vertex $i$. So $D^{2|\sigma|}_{\sigma}\times T_{\widehat{\sigma}}^{m-|\sigma|}$ can be expressed as $\mathop{\Pi} \limits _{i\in\sigma}D^{2}_{i}\times \mathop{\Pi} \limits _{j\notin\sigma}S^{1}_{j}$. Via the morphism $F=e^{i\pi t}$, the open interval $(-1,1)$ is diffemorphic to $S^{1}-\{-1\}$ ($S^{1}$ can be viewed as the unit circle in the complex plane $\mathds{C}$), which is denoted by $A$ ($S^{1}_{i}-\{-1\}$ is denoted by $A_{i}$). Without loss of generality, assume that $\sigma=\{0,m-n+1,\ldots,m-1\}$ , so its complement $\hat{\sigma}$ represents vertex set $\{1,2,\ldots,m-n\}$, $*$ is a point of $S^{1}_{m-n+1}\times\ldots\times S^{1}_{m-1}$, $y$ is the point $(1,0)$ of $S^{1}_{0}$. Out of simplicity, we will omit $*$ and $y$ in expressions, for example, we use $S^{1}_{1}\times\cdots\times S^{1}_{m-n}$ to represent the submanifold $\{y\}\times S^{1}_{1}\times\cdots\times S^{1}_{m-n}\times\{*\}$.

In this section, we will construct an isotopy inductively to move the torus
$\{y\}\times S^{1}_{1}\times\cdots\times S^{1}_{m-n}\times\{*\}$ in $\mathcal {Z}$ to the regular embedding
$T^{m-n}\subseteq D^{m-n+1}\subseteq D^{m+n}\subseteq \mathcal {Z}$.
Now we give the definition of the regular embedding $T^{m-n}\subseteq D^{m-n+1}\subseteq D^{m+n}\subseteq \mathcal {Z}$.
\subsection{Regular Embedding and Standard Torus}
We construct the regular embedding of $T^{k}$ into $\mathds{R}^{k+1}$ as follows: $S^{1}\subseteq D^{2}\subseteq \mathds{R}^{2}$,
assume that we have constructed the embedding of $T^{i-1}$ into $D^{i}\subseteq \mathds{R}^{i}$. Represent $(i+1)$-sphere as
$S^{i+1}= D^{i}\times S^{1}\bigcup S^{i-1}\times D^{2}$. By the assumption, the torus $T^{i}= T^{i-1}\times S^{1}$ can be
embedded into $D^{i}\times S^{1}$ and therefore into $S^{i+1}$. Since $T^{i}$ is compact and $S^{i+1}$ is the
one-point compactification of $\mathds{R}^{i+1}$, we have $T^{i}\subseteq\mathds{R}^{i+1}$. Inductively, we can construct
the regular embedding of $T^{k}$ into $\mathds{R}^{k+1}$ (or $D^{k+1}$).
The regular embedding of $T^{k}$ into $\mathds{R}^{n}$ is $T^{k}\subseteq \mathds{R}^{k+1}\times\{0\}\subseteq \mathds{R}^{k+1}\times\mathds{R}^{n-k-1}$,
where $T^{k}\subseteq \mathds{R}^{k+1}\times\{0\}$ is the regular embedding of $T^{k}$ into $\mathds{R}^{k+1}$.

In terms of coordinates, we can express the regular embedding torus $T^{k}$ in $\mathds{R}^{k+1}$ inductively as:
\begin{align}
T^{k}=\left\{\left(
\begin{array}{cc}\sin\alpha_{1}\cdot(1+\frac{1}{2}\sin\alpha_{2}\cdot(1+\frac{1}{2}\sin\alpha_{3}\cdot(\cdots (1+\frac{1}{2}\sin\alpha_{k})\cdots)))\\
\cos\alpha_{1}\cdot(1+\frac{1}{2}\sin\alpha_{2}\cdot(1+\frac{1}{2}\sin\alpha_{3}\cdot(\cdots (1+\frac{1}{2}\sin\alpha_{k})\cdots)))\\
\frac{1}{2}\cos\alpha_{2}\cdot(1+\frac{1}{2}\sin\alpha_{3}\cdot(\cdots (1+\frac{1}{2}\sin\alpha_{k})\cdots))\\
\vdots\\
\frac{1}{2^{k-2}}\cos\alpha_{k-1}\cdot(1+\frac{1}{2}\sin\alpha_{k})\\
\frac{1}{2^{k-1}}\cos\alpha_{k}
 \end{array}
 \right)
 |0\leq\alpha_{i}<2\pi
 \right\}
\end{align}

We  call this the standard torus $T^{k}\subseteq \mathds{R}^{k+1}$.

We shall prove that the standard torus is the regular embedding $T^{k}\subseteq \mathds{R}^{k+1}$:

\begin{lem}
If $T^{k}\subseteq \mathds{R}^{k+1}\subseteq S^{k+1}$ is the standard torus, we can express it as the embedding $T^{k}=S^{1}\times T^{k-1}\subseteq
S^{1}\times D^{k}\subseteq S^{1}\times D^{k}\cup D^{2}\times S^{k-1}\subseteq S^{k+1}$, where $T^{k-1}$ is the standard torus in $D^{k}$.
\end{lem}

\begin{proof}
We can express $S^{1}\times D^{k}\approx(sin\alpha_{0},cos\alpha_{0})\times(r_{1},\ldots,r_{k})$ in $\mathds{R}^{k+1}$ as:
\begin{align}
((1+\frac{1}{2}r_{1})sin\alpha_{0},(1+\frac{1}{2}r_{1})cos\alpha_{0},\frac{1}{2}r_{2},\ldots,\frac{1}{2}r_{k}),
\end{align}
where $-(2-\epsilon)<r_{i}<2-\epsilon$, $\epsilon>0$ is sufficiently small. Obviously, $S^{1}\times D^{k}$ embeds in $D^{k+2}$ through $S^{1}\times D^{k}=\partial D^{2}\times D^{k}\subseteq \partial (D ^{2}\times D^{k})$, so it is the first component of $S^{k+1}\approx D^{k}\times S^{1} \cup S^{k-1}\times D^{2}\supseteq{R}^{k+1}$, where $S^{k+1}$ is the one-point compactification of $\mathds{R}^{k+1}$.
While $k=1$, the standard torus $S^{1}$ is the regular embedding $S^{1}\subseteq \mathds{R}^{2}$.
Inductively suppose that the regular embedded torus $T^{k-1}\subseteq D^{k}$ can be expressed by (1). Taking this expression into (2), the regular embedded torus $T^{k}=S^{1}\times T^{k-1}\subseteq S^{1}\times D^{k}\subseteq\mathds{R}^{k+1}$ can be also expressed by (1). So by induction, we prove that the standard torus is the regular embedding $T^{k}\subseteq \mathds{R}^{k+1}$.
\end{proof}
In the next subsection, we will construct an isotopy inductively to move the torus
$S^{1}_{1}\times\cdots\times S^{1}_{m-n}$ in $\mathcal {Z}$ to the regular embedding
$T^{m-n}\subseteq A^{m-n+1}=A_{0}\times A_{1}\times\cdots\times A_{m-n}\subseteq \mathcal {Z}$. Our main idea is:
first, we can choose a facet $\sigma_{1}$ containing the vertex $1$ and construct an equivariant isotropy of $S^{1}_{1}\times\cdots\times S^{1}_{m-n}$ in $D^{2n}_{\sigma_{1}}\times T_{\widehat{\sigma_{1}}}^{m-n}$, which moves the product component $ S^{1}_{1}$ to the regular embedding $S^{1}\subseteq A^{2} \subseteq S^{1}_{0}\times S^{1}_{1}$. Inductively suppose we have constructed an equivariant isotropy to move the product component $S^{1}_{1}\times S^{1}_{2}\times\cdots\times S^{1}_{p}$ to the regular embedding $T^{p}\subseteq A^{p+1}\subseteq S^{1}_{0}\times S^{1}_{1}\times\cdots\times S^{1}_{p}$. Then we choose a facet $\sigma_{p+1}$ containing the vertex $p+1$, by Lemma 2.2 below we can construct an equivariant isotropy of $T^{p}(\subseteq A^{p+1})\times S^{1}_{p+1}\times\cdots\times S^{1}_{m-n}$ in $D^{2n}_{\sigma_{p+1}}\times T_{\widehat{\sigma_{p+1}}}^{m-n}$, which moves the product component $T^{p}\times S^{1}_{p+1}$ to the regular embedding $T^{p+1}\subseteq A^{p+2}\subseteq S^{1}_{0}\times S^{1}_{1}\times\cdots\times S^{1}_{p+1}$. By induction, we construct an isotropy of $S^{1}_{1}\times\cdots\times S^{1}_{m-n}$ in $\mathcal {Z}$ to move it to the regular embedding $T^{m-n}\subseteq A^{m-n+1}\subseteq S^{1}_{0}\times S^{1}_{1}\times\cdots\times S^{1}_{m-n}$.

The isotropy $F:T^{k}\times I\longrightarrow D^{k+2}(k\geq2)$ we construct is defined by:
\begin{align}
F(\alpha,t)=\left( \begin{array}{cc}
\frac{1}{2}\sin\alpha_{1}\cdot(1+\frac{1}{2}\sin\alpha_{2}\cdot(1+\frac{1}{2}\sin\alpha_{3}\cdot(\cdots (1+\frac{1}{2}t\sin\alpha_{k})\cdots)))\\
\frac{1}{2}\cos\alpha_{1}\cdot(1+\frac{1}{2}\sin\alpha_{2}\cdot(1+\frac{1}{2}\sin\alpha_{3}\cdot(\cdots (1+\frac{1}{2}t\sin\alpha_{k})\cdots)))\\
\frac{1}{2^{2}}\cos\alpha_{2}\cdot(1+\frac{1}{2}\sin\alpha_{3}\cdot(\cdots (1+\frac{1}{2}t\sin\alpha_{k})\cdots))\\
\vdots\\
\frac{1}{2^{k-1}}\cos\alpha_{k-1}\cdot(1+\frac{1}{2}t\sin\alpha_{k})\\
\frac{1}{2^{k}}[(1-t)\cos\alpha_{k}+t\cos(\frac{1}{2^{k}}\pi\cos\alpha_{k})]\\
\frac{1}{2^{k}}[(1-t)\sin\alpha_{k}+t\sin(\frac{1}{2^{k}}\pi\cos\alpha_{k})]
 \end{array}
 \right).
\end{align}
An examination of this isotopy proves the following :
\begin{lem}
Let $T^{k}$ be the standard torus in $D^{k+1}=D^{k}\times D^{1}\subseteq D^{k}\times S^{1}\subseteq D^{k}\times D^{2}$.
Then we may write $D^{k+2}$ as $D^{k}\times D^{2}$ so that $T^{k}$ is isotopic to
$T^{k-1}\times(\partial D^{2})$, where $T^{k-1}$ is the standard torus in $D^{k}$.
\end{lem}
\begin{proof}
While $t=0$, $F(T^{k})$ is the product $T^{k-1}\times(\partial D^{2})\subseteq D^{k}\times D^{2}$, where $T^{k-1}$ is the standard torus in $D^{k}$ and $D^{2}$ is a disk of radius $\frac{1}{2^{k}}$. While $t=1$, the last two coordinates in (3) are $(\frac{1}{2^{k}}\cos(\frac{1}{2^{k}}\pi\cos\alpha_{k}),\frac{1}{2^{k}}\sin(\frac{1}{2^{k}}\pi\cos\alpha_{k}))$. Via the morphism:
$$F=\frac{1}{2^{k}}e^{i\pi t}:(-1,1)\longrightarrow S^{1}-\{-\frac{1}{2^{k}}\},$$
the point $(\frac{1}{2^{k}}\cos(\frac{1}{2^{k}}\pi\cos\alpha_{k}),\frac{1}{2^{k}}\sin(\frac{1}{2^{k}}\pi\cos\alpha_{k}))$ can be expressed as $\frac{1}{2^{k}}\cos\alpha_{k}$ in $(-1,1)$. In this way, $T^{k}$ can be viewed as the standard torus in $D^{k+1}=D^{k}\times D^{1}\subseteq D^{k}\times S^{1}\subseteq D^{k}\times D^{2}$.
\end{proof}

\subsection{Construction of the Isotopy}
Now we can construct an isotopy of torus $T^{m-n}_{\widehat{\sigma}}\times\{0\}$ in $\mathcal {Z}$:

In $D_{\sigma}^{2n}\times T_{\widehat{\sigma}}^{m-n}$, choose a path in $D_{\sigma}^{2n}$ to connect the original point with the point $(y,*)$. Then we can construct an isotopy to move the torus $\{0\}\times T^{m-n}_{\widehat{\sigma}}\times\{0\}$ to $\{y\}\times T^{m-n}_{\widehat{\sigma}}\times\{*\}$ along the path.

As $n\geq 1$, the torus
\[S^{1}_{0}\times T^{m-n}_{\widehat{\sigma}}\subseteq \mathcal {Z}=\bigcup \limits _{\sigma \in K_{P}}D_{\sigma}^{2n}\times
T_{\widehat{\sigma}}^{m-n}\subset (D^{2})^{m}.\]
Assume the facet $\sigma_{1}$ contains vertex 1, so
\[S^{1}_{0}\times T^{m-n}_{\widehat{\sigma}}\subseteq S^{1}_{0}\times D^{2}_{1}\times S^{1}_{2}\times\ldots S^{1}_{m-n}
\subseteq D^{2n}_{\sigma_{1}}\times T_{\widehat{\sigma_{1}}}^{m-n}\subseteq \mathcal {Z}.\]
For $\forall$ $x\in S^{1}_{2}\times\ldots\times S^{1}_{m-n}$, we construct an isotopy of
$\{y\}\times S^{1}_{1}\times\{x\}$ in $S^{1}_{0}\times D^{2}_{1}\times\{x\}$
as follows:
The coordinate of a point of $\{y\}\times S_1^1$ in $A_{0}\times D_1^2 \subset S_{0}^1\times D_1^2$ is
$(1,e^{i\alpha})(0\leq\alpha< 2\pi)$. Define an isotopy:
$$F_{1}: S^{1}\times I  \longrightarrow A_{0}\times D^{2}_{1}$$
$$F_{1}(e^{i\alpha},t)=(e^{i\frac{1}{2}\pi t\sin\alpha},(1-t)e^{i\alpha}+te^{i\frac{1}{2}\pi\cos\alpha})$$
when $t=1$, $F_{1}(e^{i\alpha},1)=(e^{i\frac{1}{2}\pi\sin\alpha},e^{i\frac{1}{2}\pi\cos\alpha})$. Via the morphism:
$$F=e^{i\pi t}:(-1,1)\rightarrow A_{0}(A_{1}),$$
$F^{-1}(e^{i\frac{1}{2}\pi\sin\alpha},e^{i\frac{1}{2}\pi\cos\alpha})$ is the circle $(\frac{1}{2}\sin\alpha,\frac{1}{2}\cos\alpha)$ in $(-1,1)\times(-1,1)$, which is the standard torus in $\mathds{R}^{2}$.

In this way, for $\forall$ $x\in S^{1}_{2}\times\cdots\times S^{1}_{m-n}$, we
move $\{y\}\times S^{1}_{1}\times\{x\}$ to the regular embedding
\[
S^{1}\subset A_{0}\times A_{1}\subset S^{1}_{0}\times S^{1}_{1}\times\{x\}.
\]
By this way, we can construct an equivariant isotopy to move $\{y\}\times S^{1}_{1}\times S_2^1\times\ldots\times S^{1}_{m-n}$
into
\[
 A_{0}\times A_{1}\times (S_2^1\times\cdots \times S_{m-n}^1)\subset S^{1}_{0}\times S^{1}_{1}\times (S_2^1\times \cdots\times S_{m-n}^1).
\]

\begin{figure}[htbp]
  \centering\includegraphics[width=0.38\textwidth]{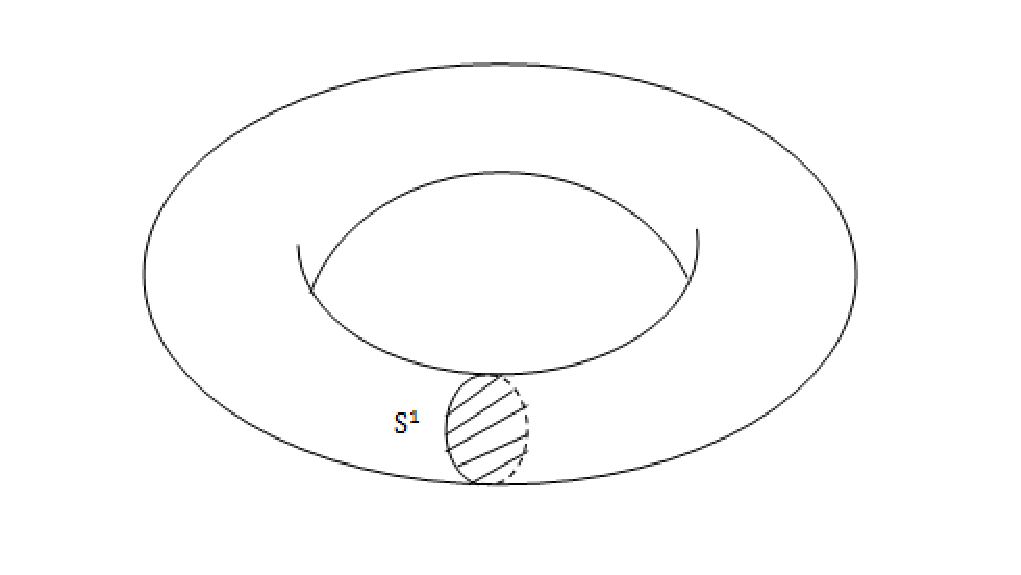}
  \caption{before the isotopy}\label{fig:digit}
  \includegraphics[width=0.35\textwidth]{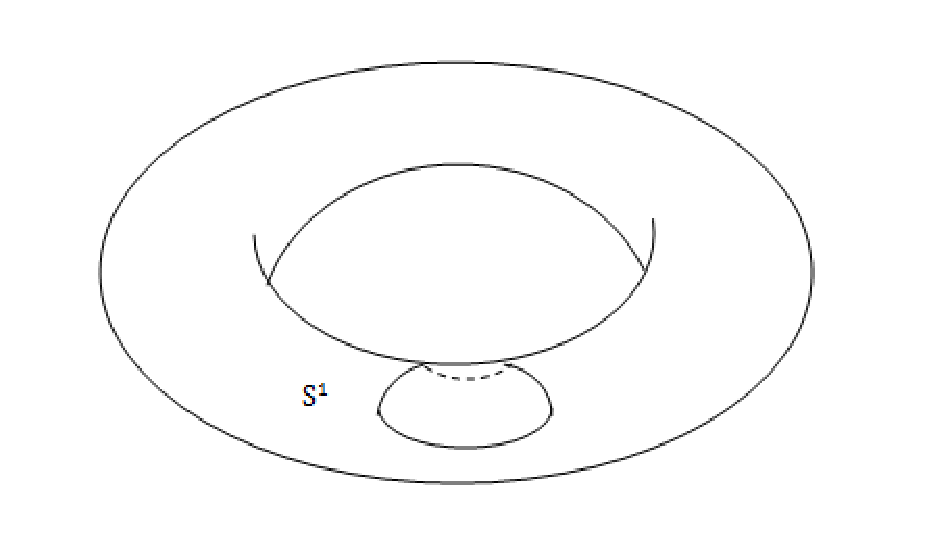}
  \caption{after the isotopy}\label{fig:digit}
\end{figure}

 Inductively suppose we have constructed an isotopy of
$\{y\}\times S^{1}_{1}\times\cdots\times S^{1}_{p}\times\{x\}$ to move it to the regular embedding
\[
T^{p}\subseteq A^{p+1}\subseteq S^{1}_{0}\times S^{1}_{1}\times\cdots\times S^{1}_{p}\times\{x\}
\]
where $x$ is a point of $S^{1}_{p+1}\times\cdots\times S^{1}_{m-n}$ and the coordinate
of the points of $T^p\subset D^{p+1}$ is expressed as (1).
Assume the facet $\sigma_{p+1}$ contains vertex $p+1$, so
\[
S^{1}_{0}\times S^{1}_{1}\times\cdots\times S^{1}_{p}\times D^{2}_{p+1}\times\{\pi(x)\}
\subset D^{2n}_{\sigma_{p+1}}\times T_{\widehat{\sigma_{p+1}}}^{m-n}\subseteq \mathcal {Z},\]
where $\pi$ is the projection
\[
\pi:S^{1}_{p+1}\times S_{p+2}^1\cdots\times S^{1}_{m-n}\rightarrow S^{1}_{p+2}\times\cdots\times S^{1}_{m-n}.
\]
Using Lemma 2.2 above, we can construct an isotopy to move the torus
\[
T^{p}(\subseteq A^{p+1}\approx (-1,1)^{p+1})\times S^{1}_{p+1}(\subseteq D^{2}_{p+1})\times\{\pi(x)\}
\]
 to the regular embedding
\[T^{p+1}\subseteq A^{p+2}\subseteq S^{1}_{0}\times S^{1}_{1}\times\ldots\times S^{1}_{p+1}\times\{\pi(x)\}.\]
So we can construct an equivariant isotropy of $T^{p}(\subseteq A^{p+1})\times S^{1}_{p+1}\times\cdots\times S^{1}_{m-n}$ in $D^{2n}_{\sigma_{p+1}}\times T_{\widehat{\sigma_{p+1}}}^{m-n}$, which moves the product component $T^{p}\times S^{1}_{p+1}$ to the regular embedding $T^{p+1}\subseteq A^{p+2}\subseteq S^{1}_{0}\times S^{1}_{1}\times\cdots\times S^{1}_{p+1}$.
By induction, we can construct an isotopy $F_{t}$ of $T^{m-n}_{\widehat{\sigma}}\times\{0\}\subseteq \mathcal {Z}$ to move it to
the regular embedding $T^{m-n}\subseteq D^{m-n+1}\subseteq D^{m+n}\subseteq \mathcal {Z}$. According to the isotopy extension theorem,
there exists an isotopy $G_{t}$ of $\mathcal {Z}$ satisfying $G_{t}\mid T^{m-n}_{\widehat{\sigma}}\times\{0\}=F_{t}$.
So the tubular neighborhood $T^{m-n}_{\widehat{\sigma}}\times D^{2n}$ of $T^{m-n}_{\widehat{\sigma}}\times\{0\}$
in $\mathcal {Z}$ is isotopic to a tubular neighborhood $N(T^{m-n})$ of $T^{m-n}$ ($T^{m-n}$ is the regular embedding
$T^{m-n}\subseteq D^{m+n}\subseteq \mathcal {Z}$). By Theorem 3.5 in \cite{[Ko]}, $N(T^{m-n})$ can be chosen to be the regular embedding
$T^{m-n}\times D^{m+n}\subseteq D^{m+n}\subseteq \mathcal {Z}$. So
$$\mathcal {Z}-T^{m-n}_{\widehat{\sigma}}\times D^{2n}\simeq \mathcal {Z}\# (S^{m+n}-T^{m-n}\times D^{2n}),$$
$$\mathcal {Z}_{v}\simeq\partial[(\mathcal {Z}\#(S^{m+n}-T^{m-n}\times D^{2n}))\times D^{2}].$$
Recall Lemma 2 in \cite{[GL]}:
\begin{lem}(Lemma 2 \cite{[GL]})
Let $M, N$ be connected $n$-manifolds, if $M$ is closed but $N$ has non-empty boundary, then
$\partial[(M\# N)\times D^{2}]$ is diffeomorphic to
$\partial[(M-D^{n})\times D^{2}]\# \partial(N \times D^{2})$.
\end{lem}
According to the lemma, $\partial[(\mathcal {Z}\#(S^{m+n}-T^{m-n}\times D^{2n}))\times D^{2}]$ is diffeomorphic to
\[\partial[(\mathcal {Z}-D^{m+n})\times D^{2}]\# \partial[(S^{m+n}-T^{m-n}\times D^{2n})\times D^{2}],\]
where torus $T^{m-n}\subseteq D^{m-n+1}\subseteq D^{m+n}\subseteq S^{m+n}$ is the regular embedding. \\

\section{The manifold $\partial[(S^{m+n}-T^{m-n}\times D^{2n})\times D^{2}]$}
 In this section, we will prove Proposition 1.2 by induction:
$$\partial[(S^{m+n}-T^{m-n}\times D^{2n})\times D^{2}]$$
is diffeomorphic to $\mathop{\#} \limits_{j=1}^{m-n} \binom{m-n}{j} (S^{j+2}\times S^{m+n-j-1})$, where $T^{m-n}\times D^{2n}\subseteq S^{m+n}$ is the regular embedding.
\begin{proof}
While $m-n=1$, the manifold $\partial[(S^{m+n}-T^{m-n}\times D^{2n})\times D^{2}]=\partial[(S^{2n+1}-S^{1}\times D^{2n})\times D^{2}]\simeq\partial(S^{2n-1}\times D^{4})\simeq S^{2n-1}\times S^{3}$.
Inductively suppose that we have proved $\partial[(S^{2n+k}-T^{k}\times D^{2n})\times D^{2}]$ is diffeomorphic to
$\mathop{\#} \limits_{j=1}^{k} \binom{k}{j} (S^{j+2}\times S^{2n+k-j-1})$.
Since $T^{k+1}\times D^{2n}\subseteq S^{2n+k+1}$ is the regular embedding, $T^{k+1}\times D^{2n}=S^{1}\times (T^{k}\times D^{2n})\subseteq S^{1}\times D^{2n+k}\subseteq S^{1}\times D^{2n+k}\cup D^{2}\times S^{2n+k-1}\approx S^{2n+k+1}$, where $T^{k}\times D^{2n}\subseteq D^{2n+k}$ is the regular embedding.
So the manifold $\partial[(S^{2n+k+1}-T^{k+1}\times D^{2n})\times D^{2}]$ is diffeomorphic to
\begin{align*}
\quad &\partial[(S^{1}\times D^{2n+k}\cup D^{2}\times S^{2n+k-1}-T^{k+1}\times D^{2n})\times D^{2}]\\
\simeq \quad &\partial[((S^{1}\times S^{2n+k}-S^{1}\times D^{2n+k}) \cup D^{2}\times S^{2n+k-1}-T^{k+1}\times D^{2n})\times D^{2}]\\
\simeq \quad&\partial[((S^{1}\times(S^{2n+k}-T^{k}\times D^{2n})-S^{1}\times D^{2n+k})\cup D^{2}\times S^{2n+k-1})\times D^{2}]\\
\simeq \quad&(\partial[S^{1}\times(S^{2n+k}-T^{k}\times D^{2n})\times D^{2}]-S^{1}\times D^{2n+k}\times S^{1})\cup D^{2}\times S^{2n+k-1}\times S^{1}\\
\simeq \quad&(S^{1}\times (\mathop{\#} \limits_{j=1}^{k} \binom{k}{j} (S^{j+2}\times S^{2n+k-j-1}))-S^{1}\times D^{2n+k}\times S^{1})\cup D^{2}\times
S^{2n+k-1}\times S^{1}\\
\simeq \quad&\partial[(\mathop{\#} \limits_{j=1}^{k} \binom{k}{j} (S^{j+2}\times S^{2n+k-j-1})-D^{2n+k}\times S^{1})\times D^{2}]\\
\simeq \quad&\partial[(\mathop{\#} \limits_{j=1}^{k} \binom{k}{j} (S^{j+2}\times S^{2n+k-j-1})\#(S^{2n+k+1}-D^{2n+k}\times S^{1}))\times D^{2}]\\
\simeq \quad&\partial[(\mathop{\#} \limits_{j=1}^{k} \binom{k}{j} (S^{j+2}\times S^{2n+k-j-1})\# S^{2n+k-1}\times D^{2})\times D^{2}]\\
\end{align*}

Recall Lemma 1 in \cite{[GL]}:
\begin{lem}(Lemma 1 \cite{[GL]})

1. If $M$ and $N$ are connected and closed n-manifolds, then $\partial[(M\# N-D^{n})\times D^{2}]$ is diffeomorphic to
$\partial[(M-D^{n})\times D^{2}]\# \partial[(N-D^{n})\times D^{2}]$.

2. $\partial[(S^{p}\times S^{q}-D^{p+q})\times D^{2}]= S^{p}\times S^{q+1}\# S^{p+1}\times S^{q}$.
\end{lem}

Using Lemma 2.3 and Lemma 3.1, $\partial[(\mathop{\#} \limits_{j=1}^{k} \binom{k}{j} (S^{j+2}\times S^{2n+k-j-1})\# S^{2n+k-1}\times D^{2})\times D^{2}]$ is diffeomorphic to

\begin{align*}
\quad &\partial[(\mathop{\#}\limits_{j=1}^{k} \binom{k}{j} (S^{j+2}\times S^{2n+k-j-1})-D^{2n+k+1})\times D^{2}]\# \partial[S^{2n+k-1}\times D^{2}\times D^{2}]\\
\simeq \quad &\mathop{\#}\limits_{j=1}^{k} \binom{k}{j} (S^{j+2}\times S^{2n+k-j})\# \mathop{\#}\limits_{j=1}^{k} \binom{k}{j} (S^{j+3}\times S^{2n+k-j-1})\# \partial[S^{2n+k-1}\times D^{2}\times D^{2}]\\
\simeq \quad &\mathop{\#}\limits_{j=1}^{k+1} \binom{k+1}{j} (S^{j+2}\times S^{2n+k-j})
\end{align*}

\end{proof}
\begin{rem}
We can define the generalized moment-angle manifold $\mathcal {Z}_{K,k}$ corresponding to the simplicial sphere $K$ by replacing the pair $(D^{2}, S^{1})$ with $(D^{k+1}, S^{k})$, All methods in this paper can be applied to the case of $(D^{k+1}, S^{k})$ and similar results will be established. Here we give the theorem analogous to Theorem 1.5 :
\begin{thm}
If $P$ is a simple polytope, $P_{v}$ is the simple polytope obtained by cutting off the vertex $v$ of $P$, then the generalized moment-angle manifold $\mathcal {Z}_{P_{v},k}$ corresponding to $P_{v}$ is diffeomorphic to
\[
\partial[(\mathcal {Z}_{P,k}-D^{n+km})\times D^{k+1}]\# \mathop{\#} \limits_{j=1}^{m-n} \binom{m-n}{j} (S^{k(j+1)+1}\times S^{k(m-j)+n-1}).
\]
\end{thm}
\end{rem}
Until now, we have proved the conjecture. In a subsequent paper we will discuss the more general problems:
the topology of the moment-angle manifold corresponding to the connected sums, bistellar moves and cutting off high dimensional faces.

\end{document}